\documentclass[draft,11pt,english]{article}

\usepackage{amsfonts,amsmath,amsthm,amscd,amssymb,latexsym,cite,verbatim,texdraw,floatflt,
caption2,pb-diagram}

\theoremstyle{plain}
\newtheorem{theorem}{Theorem}

\theoremstyle{definition}

\newcommand{\keywords}{\textbf{Key words. }\medskip}
\newcommand{\subjclass}{\textbf{MSC 2010. }\medskip}
\renewcommand{\abstract}{\textbf{Abstract. }\medskip}

\begin{document}

\sloppy

\title{Asymptotic estimates for the best uniform approximations of classes of convolution of periodic functions of high smoothness
\thanks{This work was partially supported by the Grant H2020-MSCA-RISE-2019, project number 873071 (SOMPATY: Spectral Optimization: From Mathematics to Physics and Advanced Technology).}}

\author{A.S.~Serdyuk and I.V.~Sokolenko}



\date{\ }

\maketitle

\begin{flushright}
	\textit{Dedicated to the memory of \\ Professor S.B.~Stechkin and Professor S.A.~Telyakovskii.}
\end{flushright}

\begin{abstract}
We find two-sides estimates for the best uniform approximations of classes of convolutions of $2\pi$-periodic functions from unit ball of the space  $L_p, 1 \le p <\infty,$ with fixed kernels,  modules of Fourier coefficients of which satisfy the condition $\sum\limits_{k=n+1}^\infty\psi(k)<\psi(n).$  In the case of $\sum\limits_{k=n+1}^\infty\psi(k)=o(1)\psi(n)$ the obtained estimates become the asymptotic equalities.
\end{abstract}




Let   $L_{p}$,
$1\le  p<\infty$, be the space of $2\pi$--periodic functions $f$ sum\-mable to the power $p$ on  $[-\pi,\pi)$, in which the norm is given by the formula
$$
\|f\|_{L_p}=\|f\|_{p}=\bigg(\int\limits_{-\pi}^{\pi}|f(t)|^pdt\bigg)^{1/p},
$$
 $L_{\infty}$ be the space of measurable and essentially bounded   $2\pi$--periodic functions  $f$ with the norm
$$
\|f\|_{L_\infty}=\|f\|_{\infty}=\mathop {\rm ess \sup}\limits_{t} |f(t)|,
$$
and $C$ be the space of continuous $2\pi$--periodic functions  $f$, in which the norm is defined by the equality
$$
{\|f\|_{C}=\max\limits_{t}|f(t)|}.
$$

Denote by $C^\psi_{\bar\beta,p},\ 1\le p\le\infty,$ the set of all $2\pi$-periodic functions $f$,
	representable as convolution 
\begin{equation}\label{25_7'}
f(x)=\frac{a_0}{2}+\frac{1}{\pi}\int\limits_{-\pi}^{\pi}
\varphi(x-t) \Psi_{\bar\beta}(t)dt, \ \ \ a_0\in\mathbb R, \ \ \ \varphi\in B_p^0,
\end{equation}
$$
B_p^0=\{\varphi\in L_p:\ \|\varphi\|_p\le1,\ \varphi\perp1\},
$$
with a fixed generated kernel $\Psi_{\bar\beta}\in L_{p'},\  1/p+ 1/{p'}=1,\ $ the Fourier series of which has the form
\begin{equation}\label{1*}
S[\Psi_{\bar{\beta}}](t)=\sum\limits_{k=1}^\infty \psi(k)\cos\left(kt-\frac{\beta_k\pi}2\right),\quad \beta_k\in\mathbb{R},\quad \psi(k)\ge0.
\end{equation}
A function $f$ in the representation (\ref{25_7'}) is called $(\psi,\bar{\beta})$-integral of the function $\varphi$ and  is denoted by ${\cal J}^{ \psi}_{\bar{\beta}}\varphi$  $(f={\cal J}^{ \psi}_{\bar{\beta}}\varphi)$. If $\psi(k)\neq0,\ k\in\mathbb{N},$ then the function $\varphi$ in the representation  (\ref{25_7'}) is called $(\psi,\bar{\beta})$-derivative of the function $f$ is denoted by $f^{ \psi}_{\bar{\beta}}$ $(\varphi=f^{ \psi}_{\bar{\beta}})$. The concepts of  $(\psi,\bar{\beta})$-integral and $(\psi,\bar{\beta})$-derivative  was introduced by Stepanets (see, e.g., \cite{Stepanets1987,Stepanets2005}). 
Since \mbox{$\varphi\in L_p$} and $\ \Psi_{\bar\beta}\in L_{p'},$ then (see. \cite[Proposition 3.9.2.]{Stepanets2005}) the function $f$ of the form \eqref{25_7'} is a continuous function, i.e. $C^\psi_{\bar\beta,p}\subset C.$

In the case $\beta_k\equiv\beta,\ \beta\in \mathbb R,$ the classes $C^\psi_{\bar\beta,p}$ are denoted by $C^\psi_{\beta,p}$. 

For $\psi(k)=k^{-r}, r>0, $ the classes  $C^\psi_{\bar\beta,p}$ та $C^\psi_{\beta,p}$ are denoted by $W^r_{\bar\beta,p}$ and $W^r_{\beta,p}$, respectively. The $W^r_{\beta,p}$ are  the well-known  Weyl-Nagy classes (see, e.g., \cite{Sz.-Nagy1938, Stechkin1956_2, Stepanets1987, Stepanets2005}). In other words $W^r_{\beta,p}, 1\le p\le \infty,$ are the classes of  $2\pi$-periodic functions $f$, representable
as convolutions of the form
\begin{equation}\label{1}
f(x)=\frac{a_0}{2}+\frac{1}{\pi}\int\limits_{-\pi}^{\pi}
\varphi(x-t) B_{r,\beta}(t)dt, \ \ \ a_0\in\mathbb R, 
\end{equation}
the Weyl-Nagy kernels $B_{r,\beta}(\cdot)$ of the form
\begin{equation}\label{2*}
B_{r,\beta}(t)=\sum\limits_{k=1}^\infty k^{-r}\cos\left(kt-\frac{\beta\pi}2\right),\quad r>0,\quad \beta\in\mathbb{R},
\end{equation}
with functions $\varphi \in B_p^0$. The function $\varphi$ in the formula \eqref{1} is called  the Weyl-Nagy derivative of the function $f$ and is denoted by $f_\beta^r$.

If  $r\in\mathbb N$ and $ \beta=r,\ $ then the functions $B_{r,\beta}(\cdot)$ of the form (\ref{2*})  are  the well-known Bernoulli kernels and  the corresponding classes $W^r_{\beta,p}$ coincide with the well-known classes $W^r_{p}$, which consist of $2\pi$-periodic functions $f$  with absolutely continuous derivatives $f^{(k)}$ up to $(r-1)$-th order inclusive and such that  $\|f^{(r)}\|_p\le1$. In addition,   for almost everywhere $x\in\mathbb{R} \ \ f^{(r)}(x)=f_r^r(x)=\varphi(x),\ $  where $\varphi$ is the function from (\ref{1}).

For $\psi(k)=e^{-\alpha k^{r}},\ \alpha>0, \ r>0, $   the classes  $C^\psi_{\bar\beta,p}$ are denoted by $C^{\alpha,r}_{\bar\beta,p}.$  In the case of $r=1$, $\beta_k\equiv\beta,\ \beta\in \mathbb R,$ and $p=\infty$ the sets $C^{\alpha,r}_{\bar\beta,p}$ are  well-known classes of the Poisson integrals  $C^{\alpha,1}_{\beta,\infty}$ (see, e.g., \cite{Stechkin1980, Stepanets1987, Stepanets2005}).

If $f\in C$  by $E_n(f)_C$ we denote the best uniform approximation of the function $f$ by elements of the subspace ${\cal T}_{2n-1}$ of trigonometric polynomials $T_{n-1}$ of the order $n-1$: 
$$
T_{n-1}(x)=\frac{\alpha_0}{2}+\sum\limits_{k=1}^{n-1}
(\alpha_k\cos kx+\beta_k\sin kx),\ \ \  \alpha_k,\beta_k\in \mathbb{R}.
$$

Let $\mathfrak N$ be the some functional class from the space $C$ $(\mathfrak N\subset C)$. Then the quantity  
\begin{equation}\label{15_3'}
E_n(\mathfrak N)_C=\sup\limits_{f\in \mathfrak N} E_n(f)_C= \sup\limits_{f\in \mathfrak N}\inf\limits_{T_{n-1}\in{\cal T}_{2n-1}} \|f(\cdot)-T_{n-1}(\cdot)\|_C
\end{equation}
is called the best uniform approximation of the class $\mathfrak{N}$ by elements of the
subspace ${\cal T}_{2n-1}$ of trigonometric polynomials $T_{n-1}$ of the order $n-1$.

At present, the exact values for the quantities of the form  \eqref{15_3'} are known  for important functional classes $\mathfrak{N}$. In particular, thanks to the articles of  Favard \cite{Favard_1936, Favard_1937}, Akhiezer and Krein \cite{Akhiezer_Krein1937}, Nikol'skii \cite{Nikol'skii1946}, Dzyadyk \cite{Dzyadyk_1959,Dzyadyk_1974}, Stechkin \cite{Stechkin1956_2} 
and  Sun \cite{Sun_1961} the exact values of the best uniform approximations of the Weyl-Nagy classes $W_{\beta, \infty}^r$ are found  for arbitrary $ r>0$ and $\beta \in \mathbb{R}.$

For the classes of the Poisson integrals $C^{\alpha,1}_{\beta,\infty}$ the exact values of the form \eqref{15_3'} are also known for all $\alpha>0$ and $\beta\in\mathbb{R}$  thanks to the articles of Krein \cite{Krein_1938}, Bushanskij \cite{Bushanskij} and Shevaldin \cite{Shevaldin_1992} (see also \cite{Stepanets2005,Baraboshkina2011}). 

The exact values of the best approximations $E_n(\mathfrak N)_C$ were obtained in a number of other cases (see, e.g.,  \cite{Pinkus1985,Serdyuk1995,Serdyuk1998,Serdyuk1999,Serdyuk_2002_zb,Serdyuk_2005_7,Stepanets2005}).

In the general case, the problem of finding of the exact values of the best uniform approximations of the classes  $C^\psi_{\bar\beta,p}$ for $ 1\le p\le\infty$ remains open, and therefore, the investigation of the asymptotic behavior of the quantities $E_n(C^\psi_{\bar\beta,p})_C$  as $n\rightarrow\infty$ is relevant. 

In this paper we investigate the problem of finding of the asymptotic equalities for the quantities \eqref{15_3'} as $n\rightarrow\infty$ in the case, when the classes $\mathfrak N$ are the classes $C^\psi_{\bar\beta,p},\ 1\le p\le\infty,$ and the sequences $\psi(k)$ decrease to zero very rapidly, in particular, when
\begin{equation}\label{15_4'}
\sum\limits_{k=n+1}^\infty \psi(k)=o(1)\psi(n).
\end{equation}

This work can be considered a continuation of the authors' research  \cite{Serdyuk_2004,Serdyuk_Sokolenko2011,Serdyuk_Sokolenko2016}, in which the asymptotics of the best uniform approximations of classes of the generalized $(\psi,\bar{\beta})$-integrals were investigated.

Note that in the case of $p=\infty$ the asymptotic equalities and even the exact values of the quantities  $E_{n}(C_{\bar\beta,\infty}^\psi)_{C}$ are known  under certain restrictions on $\psi(k)$  (see, e.g.,  \cite{Serdyuk_Stepanets2000_03,Serdyuk_2005_7}).

For a fixed $\mathfrak N\subset C$    denote by ${\cal E}_n(\mathfrak N)_C$  the quantity
\begin{equation}\label{15_4}
{\cal E}_n(\mathfrak N)_C=\sup\limits_{f\in \mathfrak N}\|f(\cdot)-S_{n-1}(f;\cdot)\|_C,
\end{equation}
where    $S_{n-1}(f;\cdot)$ is the partial Fourier sum of order $n-1$ of the function $f$.

Since
\begin{equation}\label{15_5}
E_n(\mathfrak N)_C\le{\cal E}_n(\mathfrak N)_C, \quad \mathfrak{N}\subset C,
\end{equation}
then the quantities \eqref{15_4} are naturally used for upper bounds for the best approximations of the classes $\mathfrak{N}.$

The problem of finding of the asymptotic equalities for the quantities of the form \eqref{15_4} as $n\rightarrow\infty$ is called the Kolmogorov–Nikol’skii problem for the Fourier sums. The Kolmogorov–Nikol’skii problem has a rich history. Reviews on the history of this problem
can be found e.g. in the monographs \cite{Stepanets1987,Stepanets2005}.

For characteristics of the form (\ref{15_4}) on the Weyl-Nagy classes $W^r_{\beta,\infty}$ $(\mathfrak  N=W^r_{\beta,\infty})$ the following asymptotic formula holds 
\begin{equation}\label{25_4}
{\cal E}_{n}(W^r_{\beta,\infty})_{C}=\frac4\pi\frac{\ln n}{n^rзнам}+{\cal O}\left(\frac1{n^r}\right),\quad r>0,\quad\beta\in\mathbb R.
\end{equation}
For $r\in\mathbb{N}$ and $\beta=r$   this estimate was obtained by Kolmogorov \cite{Kolmogorov1985}, for arbitrary $r>0$  by Pinkevich \cite{Pinkevich1940} and Nikol'skii \cite{Nikol'skii1941}. In the general case the estimate (\ref{25_4}) follows from results of Efimov \cite{Efimov1960} and Telyakovskii \cite{Telyakovskii1961}. 

In these works the parameters $r$ and $\beta$ of the Weyl-Nagy classes were assumed to be
fixed, and the question about the dependence of the remainder term in the estimates (\ref{25_4}) on these parameters was not considered.
 The character of the dependence on $r$ and $\beta$ of the remainder term in the estimate (\ref{25_4}) was
 investigated by Sokolov, Selivanova, Natanson, Telyakovskii, Stechkin and other authors  (see  \cite{Stechkin1980, Telyakovskii1989, Telyakovskii1994, Serdyuk_Sokolenko2019} and the references therein).

In the work of Stechkin \cite{Stechkin1980} the asymptotic behavior of the quantities ${\cal E}_{n}(W^r_{\beta,\infty})_{C}$  was completely investigated as $n\rightarrow\infty$ and $r\rightarrow\infty$. Namely, he proved that for
arbitrary $r\ge1$ and $\beta\in\mathbb{R}$ the following equality takes place
\begin{equation}\label{25_6}
{\cal E}_{n}(W^r_{\beta,\infty})_{C}=\frac1{n^r}\left(\frac8{\pi^2}\mathbf{K}(e^{-r/n})+{\cal O}(1)\frac1r\right),
\end{equation}
where
\begin{equation}\label{25_7}
\mathbf{K}(q)=\int\limits_{0}^{\pi/2}\frac{dt}{\sqrt{1-q^2\sin^2t}}
\end{equation}
is a complete elliptic integral of the first kind, and ${\cal O}(1)$ is a quantity uniformly bounded
with respect to $r, n$ and $\beta$.

In addition,  Stechkin \cite[Theorem 4]{Stechkin1980} proved that for rapidly growing $r$ the remainder in the equality (\ref{25_6}) can be improved. Namely, for arbitrary $r\ge n+1$ and $\beta\in\mathbb{R}$ the following formula holds
\begin{equation}\label{25_8}
{\cal E}_{n}(W^r_{\beta,\infty})_{C}=
\frac1{n^r}\left(\frac4{\pi} +{\cal O}(1)\left(1+\frac1n\right)^{-r} \right),
\end{equation}
where ${\cal O}(1)$ is a quantity uniformly bounded with respect to $r, n$ and $\beta$. If $r/n\rightarrow\infty$ then the estimate (\ref{25_8}) becomes the asymptotic equality.

In the works of Telyakovskii \cite{Telyakovskii1989,Telyakovskii1994} it was shown  that  the second term in the formula (\ref{25_8}) can be replaced by a smaller one, namely, we can write $\displaystyle {\cal O}(1)(1+ 2/n)^{-r}$ instead of $\displaystyle {\cal O}(1)(1+ 1/n)^{-r}$, and it s also the estimate (\ref{25_8}) sharper by separating out the following terms of the asymptotics.

In the work of authors \cite{Serdyuk_Sokolenko2019}, in particular, the formula of Stechkin  \eqref{25_8}  was generalized on classes $W^r_{\bar\beta,p}.$  Namely, it is proved that if $\ 1 \le p \le \infty,\ n\in\mathbb{N}$ and $\ \bar{\beta}=\{\beta_k\}_{k=1}^\infty$ is arbitrary sequence of real numbers, then for $r\ge n+1$ the following estimate holds	
\begin{equation}\label{06_1t1}
{\cal E}_{n}(W^r_{\bar\beta,p})_{C}=
n^{-r}\Bigg(\frac{\|\cos t\|_{p'}}{\pi} +{\cal O}(1)  \bigg(1+\frac1n\bigg)^{-r} \Bigg), \ \ \ \frac1p+\frac1{p'}=1,
\end{equation}
where ${\cal O}(1)$ is a quantity uniformly bounded in all parameters.

For  the classes of the generalized Poisson integrals $C_{\beta,p}^{\alpha,r}$ and for all fixed $\alpha>0, r>0, \beta\in\mathbb{R}$ and $1\le p\le\infty$ the asymptotic equalities for the  quantities \eqref{15_4} as $n\rightarrow\infty$ are known due to works of  Nikol'skii \cite{Nikol'skii1946}, Stechkin \cite{Stechkin1980}, Stepanets \cite{Stepanets1987,Stepanets2005}, Telyakovskii \cite{Telyakovskii1989,Telyakovskii1994}, Serdyuk and Stepanyuk  \cite{Serdyuk_Stepanyuk2019}. 

As for the classes $C_{\bar\beta,p}^\psi$, for the rapidly decreasing sequences $\psi(k)$ (in particular, when the condition \eqref{15_4'} is satisfied) the asymptotics of the quantities  ${\cal E}_{n}(C_{\bar\beta,p}^\psi)_{C}$  are known for all $1\le p\le\infty $ and $\beta_k\in\mathbb{R}$ (see \cite{Stepanets1987,Stepanets2005,Telyakovskii1989,Telyakovskii1994,Serdyuk_2005_8,Serdyuk_Stepanyuk2019}). Therefore, the main goal of this work is focused on finding lower estimates for the best approximations $E_{n}(C_{\bar\beta,p}^\psi)_{C}, 1\le p\le\infty.$ 

In present paper we find two-sides estimates for the quantities $E_{n}(C_{\bar\beta,p}^\psi)_{C}$ and ${\cal E}_{n}(C_{\bar\beta,p}^\psi)_{C}$, from which, in particular, it follows that under condition \eqref{15_4'}
\begin{equation}\label{15_6}
E_{n}(C_{\bar\beta,p}^\psi)_{C}\sim{\cal E}_{n}(C_{\bar\beta,p}^\psi)_{C}\sim\frac{\|\cos t\|_{p'}}{\pi}\psi(n),
\end{equation}
where $1\le p\le\infty,\ $  $\beta_k\in\mathbb{R}$, $\displaystyle\frac1p+\frac1{p'}=1,\ $  and 
$A(n){\sim} B(n)$ as $n\rightarrow\infty$ means that $\lim\limits_{n\rightarrow\infty} {A(n)}/{B(n)}{=}1.$

It should be noted that the asymptotic equality \eqref{15_6}  is a manifestation of the high-smoothness effect, which occurs when the first harmonic in the remainder of the Fourier series after the $(n-1)$-th member of the generating kernels $\Psi_{\bar{\beta}}(t)$  of the form \eqref{1*} is a dominant in estimating of the $L_{p'}$-norm of the specified remainder of the series.
When the sequences  $\psi(k)$ decrease to zero not so fast,the equalities $E_{n}(C_{\bar\beta,p}^\psi)_{C}$ and ${\cal E}_{n}(C_{\bar\beta,p}^\psi)_{C}$ for $1\le p<\infty$ do not asymptotically coincide with each other and for $p=\infty$ can be different even in order.
 

 Let $n\in\mathbb{N}$. In what follows, we will require that the sequence of modules of the Fourier coefficients of the generated kernel  $\Psi_{\bar{\beta}}(t)$ satisfies the condition 
\begin{equation}\label{2}
\sum\limits_{k=n+1}^\infty \psi(k)<\psi(n).
\end{equation}


\begin{theorem}\label{1t}
	For arbitrary $\{\beta_k\}_{k=1}^\infty,$ $\beta_k\in\mathbb R,$ $1 < p \le \infty,$   $n\in\mathbb{N}$ and $\psi(k)\ge0$, which satisfy the condition \eqref{2}, the following inequality holds
	\begin{equation}\label{3}
	E_{n}(C_{\bar\beta,p}^\psi)_{C}\ge 
	\frac{\|\cos t\|_{p'}}{\pi}\left(\psi(n)-\sum\limits_{k=n+1}^\infty \psi(k)\right),
	\end{equation}
	where $\displaystyle\frac1p+\frac1{p'}=1.$
\end{theorem}

\begin{proof}
Let us consider a function
\begin{equation}\label{4}
\varphi_{n,p}(t)=\|\cos t\|_{p'}^{1-p'}|\cos nt|^{p'-1}{\rm sign} \cos nt.
\end{equation}
Since $\varphi_{n,p}\perp1$ and for $1<p<\infty$
\begin{equation*}\label{5}
\left(\int\limits_{-\pi}^\pi
|\varphi_{n,p}(t)|^p dt\right)^\frac1p=\|\cos t\|_{p'}^{1-p'}\left(\int\limits_{-\pi}^\pi
|\cos nt|^{p(p'-1)} dt\right)^\frac1p
=
$$
$$
=
\|\cos t\|_{p'}^{1-p'}\left(\int\limits_{-\pi}^\pi|\cos nt|^{p'} dt\right)^\frac1p=
$$
$$
=
\|\cos t\|_{p'}^{1-p'}\left(\left(\int\limits_{-\pi}^\pi|\cos nt|^{p'} dt\right)^\frac1{p'}\right)^\frac{p'}{p}
= \|\cos t\|_{p'}^{1-p'}\|\cos t\|^\frac{p'}{p}_{p'}=1,
\end{equation*}
then $\|\varphi_{n,p}\|_p=1$ and therefore $\varphi_{n,p}\in B_p^0,\ 1<p<\infty.$ 

We also put $\varphi_{n,\infty}(t)={\rm sign}\cos nt$.	Obviously that $\varphi_{n,\infty}\in B_\infty^0$.

Let us consider a function
$$
f_{n,p,\bar{\beta}}(x)=\frac{1}{\pi}\int\limits_{-\pi}^{\pi}
\varphi_{n,p}(x-t)\Psi_{\bar{\beta}}(t)dt
$$ 
from the class $C_{\bar{\beta},p}^\psi, 1<p\le\infty$. We have
\begin{equation}\label{5}
f_{n,p,\bar{\beta}}(x)=\frac{1}{\pi}\int\limits_{-\pi}^{\pi}\varphi_{n,p}(x-t)
\sum_{k=1}^\infty\psi(k)\cos\left(kt-\frac{\beta_k\pi}2\right)dt
=
$$
$$
=
\frac{1}{\pi}\int\limits_{-\pi}^{\pi}\varphi_{n,p}(x-t)
\sum_{k=1}^{n-1}\psi(k)\cos\left(kt-\frac{\beta_k\pi}2\right)dt+
$$
$$
+
\frac{1}{\pi}\int\limits_{-\pi}^{\pi}\varphi_{n,p}(x-t)
\sum_{k=n}^\infty\psi(k)\cos\left(kt-\frac{\beta_k\pi}2\right)dt.
\end{equation}	

Since $\varphi_{n,p}(t)$ is the $\displaystyle\frac{2\pi}n$-periodic function, then according to  \cite[Proposition 4.1.2]{Kornejchuk1987} it is orthogonal to all trigonometric polynomials of order $n-1$. Therefore, from (\ref{5}) we obtain
\begin{equation}\label{6}
f_{n,p,\bar{\beta}}(x)=
\frac{1}{\pi}\int\limits_{-\pi}^{\pi}\varphi_{n,p}(x-t)
\sum_{k=n}^\infty\psi(k)\cos\left(kt-\frac{\beta_k\pi}2\right)dt=
$$
$$
=
\frac{1}{\pi}\int\limits_{-\pi}^{\pi}\varphi_{n,p}(t)
\psi(n)\cos\left(n(x-t)-\frac{\beta_n\pi}2\right)+
$$
$$
+
\frac{1}{\pi}\int\limits_{-\pi}^{\pi}\varphi_{n,p}(x-t)
\sum_{k=n+1}^\infty\psi(k)\cos\left(kt-\frac{\beta_k\pi}2\right)dt=:F_1(x)+F_2(x).
\end{equation}	

By virtue of H$\rm\ddot{o}$lder's inequality and the fact that $\varphi_{n,p}\in B_p^0, 1<p\le\infty,$ we obtain the estimate
\begin{equation}\label{7}
\left\|F_2\right\|_C 
\le 
\frac{1}{\pi}\|\varphi_{n,p}\|_p\left\|\sum_{k=n+1}^\infty\psi(k)\cos\left(kt-\frac{\beta_k\pi}2\right)\right\|_{p'}=
$$
$$
=
\frac{1}{\pi}\left\|\sum_{k=n+1}^\infty\psi(k)\cos\left(kt-\frac{\beta_k\pi}2\right)\right\|_{p'}
\le
\frac{\|\cos t\|_{p'}}{\pi} \sum_{k=n+1}^\infty\psi(k).
\end{equation}

Taking into account that
$$
F_1(x)
=
$$
$$
=
\frac{1}{\pi}\int\limits_{-\pi}^{\pi}
\|\cos t\|_{p'}^{1-p'}|\cos nt|^{p'-1}{\rm sign} \cos nt\,\psi(n)\cos\left(n(x{-}t){-}\frac{\beta_n\pi}2\right) dt=
$$
$$
=
\frac{\|\cos t\|_{p'}^{1-p'}}{\pi}\psi(n)\int\limits_{-\pi}^{\pi}|\cos nt|^{p'-1}{\rm sign} \cos nt
\times
$$
$$
\times
\left(\cos\left(nx-\frac{\beta_n\pi}2\right)\cos nt+\sin\left(nx-\frac{\beta_n\pi}2\right)\sin nt\right) dt,
$$	
we consider on the period $\displaystyle\left[\frac{\beta_n\pi}{2n}, 2\pi+\frac{\beta_n\pi}{2n}\right)$ the set of $2n$ points $\displaystyle\frac{\beta_n\pi}{2n}=x_0<x_1<\ldots<x_{2n-1}<2\pi+\frac{\beta_n\pi}{2n}$ of the form
\begin{equation}\label{8}
x_m=\frac{\beta_n\pi}{2n}+\frac{m\pi}{nзнам},\quad m=\overline{0,2n-1}.
\end{equation}
Let us show that
\begin{equation}\label{9}
F_1(x_m)=\frac{(-1)^m\|\cos t\|_{p'}}\pi \psi(n), \quad m=\overline{0,2n-1}.
\end{equation}

Indeed, as for any $t\in[-\pi,\pi)$
$$
|\cos nt|^{p'-1}{\rm sign} \cos nt\cos\left(n(x_m-t)-\frac{\beta_n\pi}2\right)
=
$$
$$
=|\cos nt|^{p'-1}{\rm sign} \cos nt \cos\left(\frac{\beta_n\pi}2 +m\pi -nt-\frac{\beta_n\pi}2\right)=
$$
$$
=(-1)^m |\cos nt|^{p'-1}{\rm sign} \cos nt\, \cos nt=(-1)^m|\cos nt|^{p'},
$$
then
$$
F_1(x_m)
=
$$
$$
=\frac{\|\cos t\|_{p'}^{1-p'}}{\pi}\psi(n)\int\limits_{-\pi}^{\pi}
|\cos nt|^{p'-1}{\rm sign} \cos nt \cos\left(n(x_m{-}t){-}\frac{\beta_n\pi}2\right)dt=
$$
$$
=\frac{\|\cos t\|_{p'}^{1-p'}}{\pi}\psi(n)(-1)^m\int\limits_{-\pi}^{\pi}
|\cos nt|^{p'}dt=\frac{(-1)^m\|\cos t\|_{p'}}{\pi}\psi(n).
$$
From \eqref{2}, (\ref{7}) and (\ref{9}) it follows that
$$
|F_1(x_m)|=\frac{\|\cos t\|_{p'}}{\pi}\psi(n)>\frac{\|\cos t\|_{p'}}{\pi}\sum_{k=n+1}^\infty\psi(k)\ge\|F_2\|_C\ge |F_2(x_m)|,
$$
and hence for $f_{n,p,\bar{\beta}}(x_m)=F_1(x_m)+F_2(x_m)$ the following relations hold
\begin{equation}\label{10}
{\rm sign}f_{n,p,\bar{\beta}}(x_m)={\rm sign}F_1(x_m)=(-1)^m,\quad m=0,1,\ldots,2n-1,
\end{equation}
and
\begin{equation}\label{11}
|f_{n,p,\bar{\beta}}(x_m)|\ge|F_1(x_m)|-|F_2(x_m)|\ge|F_1(x_m)|-\|F_2\|_C\ge
$$
$$
\ge\frac{\|\cos t\|_{p'}}{\pi}\Big(\psi(n)-\sum\limits_{k=n+1}^\infty \psi(k)\Big).
\end{equation}
Then by virtue of Valle Poussin's theorem (see, i.e., \cite[Theorem 6.2.2]{Stepanets1987})
$$
E_n(f_{n,p,\bar{\beta}})_C\ge\min_{m=0,1,\ldots,2n-1}|f_{n,p,\bar{\beta}}(x_m)|\ge\frac{\|\cos t\|_{p'}}{\pi}\Big(\psi(n)-\sum\limits_{k=n+1}^\infty \psi(k)\Big).
$$
\end{proof}

\begin{theorem}\label{1.5t}
	For arbitrary $\{\beta_k\}_{k=1}^\infty,$ $\beta_k\in\mathbb R,$  $n\in\mathbb{N}$ and $\psi(k)\ge0$, which satisfy the condition \eqref{2}, the following inequality holds
	\begin{equation}\label{3'}
	E_{n}(C_{\bar\beta,1}^\psi)_{C}\ge 
	\frac{1}{\pi}\left(\psi(n)-\sum\limits_{k=n+1}^\infty \psi(k)\right).
	\end{equation}
\end{theorem}

\begin{proof}  
Let us consider a function
\begin{equation}\label{4'}
\varphi_{n,1}(\delta,t)=\left\{\begin{array}{cl}
\displaystyle\frac{(-1)^m}{2\delta}, & \displaystyle t\in\left(\frac{m\pi}{n}-\frac{\delta}{2n},\frac{m\pi}{n}+\frac{\delta}{2n}\right),\ m\in\mathbb{Z},\\ \\ 
0,&\displaystyle t\in\mathbb{R}\setminus\mathop{\cup}\limits_{m\in\mathbb{Z}}\left(\frac{m\pi}{n}-\frac{\delta}{2n},\frac{m\pi}{n}+\frac{\delta}{2n}\right),\ m\in\mathbb{Z},
\end{array}
\right.
\end{equation}
where $0<\delta<\frac\pi2$.

According to the definition 
\begin{equation}\label{4''}
{\rm sign }\,\varphi_{n,1}(\delta,t)={\rm sign }\,\cos nt, \quad\displaystyle t\in\left(\frac{m\pi}{n}-\frac{\delta}{2n},\frac{m\pi}{n}+\frac{\delta}{2n}\right),\ m\in\mathbb{Z},
\end{equation}
and the function $\varphi_{n,1}(\delta,t) $ is $\displaystyle\frac{2\pi}n$-periodic function. Besides
\begin{equation}\label{4'}
\|\varphi_{n,1}(\delta,\cdot)\|_1=\sum_{m=0}^{2n-1}\frac1{2\delta}\frac{2\delta}{2n}=1,
\end{equation}
that is for any  $0<\delta<\frac\pi2\ $
$\ \varphi_{n,1}(\delta,\cdot)\in B_1^0.$

Let us consider a function
\begin{equation}\label{5'}
f_{n,1,\bar{\beta}}(\delta,x)=\frac{1}{\pi}\int\limits_{-\pi}^{\pi}
\varphi_{n,1}(\delta,x-t)\Psi_{\bar{\beta}}(t)dt
\end{equation}
from the class $C_{\bar{\beta},1}^\psi$. Since $\varphi_{n,1}(\delta,t) \perp\sum\limits_{k=1}^{n-1}\psi(k)\cos\left(kt-\frac{\beta_k\pi}2\right)$, then
\begin{equation*}\label{5'''}
f_{n,1,\bar{\beta}}(\delta,x)=F_1(\delta,x)+F_2(\delta,x),
\end{equation*}
where
\begin{equation}\label{6'}
F_1(\delta,x)=
\frac{1}{\pi}\int\limits_{-\pi}^{\pi}\varphi_{n,1}(\delta,x-t)
\psi(n)\cos\left(nt-\frac{\beta_n\pi}2\right)dt=
$$
$$
=
\frac{1}{\pi}\psi(n)\int\limits_{-\pi}^{\pi}\varphi_{n,1}(\delta,t)
\cos\left(n(x-t)-\frac{\beta_n\pi}2\right)dt,
\end{equation}	
\begin{equation}\label{7'}
F_2(\delta,x)=
\frac{1}{\pi}\int\limits_{-\pi}^{\pi}\varphi_{n,1}(\delta,x-t)
\sum_{k=n+1}^\infty\psi(k)\cos\left(kt-\frac{\beta_k\pi}2\right)dt.
\end{equation}	

By virtue of H$\rm\ddot{o}$lder's inequality and the equality \eqref{4'}
\begin{equation}\label{8'}
\left\|F_2(\delta,\cdot)\right\|_C \le 
\frac{1}{\pi}\|\varphi_{n,1}(\delta,\cdot)\|_1\left\|\sum_{k=n+1}^\infty\psi(k)\cos\left(kt-\frac{\beta_k\pi}2\right)\right\|_{\infty}
\le
$$
$$
\le
\frac{1}{\pi} \sum_{k=n+1}^\infty\psi(k).
\end{equation}

Let us consider on the segment $\displaystyle\left[\frac{\beta_n\pi}{2n}, 2\pi+\frac{\beta_n\pi}{2n}\right)$ the following point set
\begin{equation*}\label{8}
x_m=\frac{\beta_n\pi}{2n}+\frac{m\pi}{n},\quad m=\overline{0,2n-1}.
\end{equation*}

Let us show that
\begin{equation}\label{9'}
F_1(\delta,x_m)=\frac{(-1)^m}\pi \frac2\delta\sin\frac{\delta}{2}\,\psi(n), \quad m=\overline{0,2n-1}.
\end{equation}

Indeed, since 
$$
\varphi_{n,1}(\delta,t)\cos\left(n(x_m-t)-\frac{\beta_n\pi}2\right)=
$$
$$
=
\varphi_{n,1}(\delta,t)\cos\left(n\left(\frac{\beta_n\pi}{2n} +\frac{m\pi}n -t\right)-\frac{\beta_n\pi}2\right)=
$$
$$
= \varphi_{n,1}(\delta,t)\cos(m\pi-nt)=(-1)^m\varphi_{n,1}(\delta,t)\cos nt,
$$
Then, in view of  \eqref{4''}, \eqref{4'} and \eqref{6'},
$$
F_1(x_m)
=\frac1{\pi}\psi(n)\int\limits_{-\pi}^{\pi}
(-1)^m\varphi_{n,1}(\delta,t)\cos ntdt
=
$$
$$
=
\frac{(-1)^m }{\pi}\psi(n)\sum_{m=0}^{2n-1}\int\limits_{-\frac\pi{2n}+\frac {m\pi}{2n}}^{\frac\pi{2n}+\frac {m\pi}{2n}}
\varphi_{n,1}(\delta,t)\cos ntdt=
$$
$$
=\frac{(-1)^m }{\pi}\psi(n)\sum_{m=0}^{2n-1}\int\limits_{-\frac\delta{2n}+\frac {m\pi}{2n}}^{\frac\delta{2n}+\frac {m\pi}{2n}}
\frac{(-1)^m }{2\delta}\cos nt dt
=
$$
$$
=
\frac{(-1)^m }{\pi}\psi(n)\sum_{m=0}^{2n-1}\, \frac{1 }{2\delta} \int\limits_{-\frac\delta{2n}}^{\frac\delta{2n}}
\cos nt dt=
$$
$$
=\frac{(-1)^m }{\pi}\psi(n)\frac{2n}{2\delta n}\int\limits_{-\frac\delta2}^{\frac\delta2}
\cos t dt=\frac{(-1)^m}\pi \psi(n)\frac2\delta\sin\frac{\delta}{2}.
$$

Choose $\delta$  so small that 
\begin{equation}\label{10'}
\frac2\delta\sin\frac{\delta}{2}\,\psi(n)>\sum_{k=n+1}^\infty\psi(k).
\end{equation}

For such $\delta$, taking into account  \eqref{8'}, \eqref{9'} and \eqref{10'},
$$
|F_1(\delta,x_m)|=\frac2\delta\sin\frac{\delta}{2} \psi(n)>\frac1{\pi}\sum_{k=n+1}^\infty\psi(k)\ge\|F_2(\delta,\cdot)\|_C\ge |F_2(\delta,x_m)|,
$$
and hence for $f_{n,p,\bar{\beta}}(\delta,x_m)=F_1(\delta,x_m)+F_2(\delta,x_m)$ the following relations hold
\begin{equation}\label{11'}
{\rm sign}f_{n,p,\bar{\beta}}(\delta,x_m)={\rm sign}F_1(\delta,x_m)=(-1)^m,\quad m=0,1,\ldots,2n-1,
\end{equation}
and
\begin{equation}\label{12'}
|f_{n,p,\bar{\beta}}(\delta,x_m)|\ge|F_1(\delta,x_m)|-|F_2(\delta,x_m)|\ge|F_1(\delta,x_m)|-\|F_2
(\delta,\cdot)\|_C\ge
$$
$$
\ge\frac{1}{\pi}\Big(\frac2\delta\sin\frac{\delta}{2} \psi(n)-\sum\limits_{k=n+1}^\infty \psi(k)\Big).
\end{equation}
Then by virtue of Valle Poussin's theorem (see, i.e., \cite[Theorem 6.2.2]{Stepanets1987})
\begin{equation}\label{13'}
E_n(f_{n,p,\bar{\beta}}(\delta,\cdot))_C\ge\min_{m=0,1,\ldots,2n-1}|f_{n,p,\bar{\beta}}(\delta,x_m)|\ge
$$
$$
\ge\frac{1}{\pi}\Big(\frac2\delta\sin\frac{\delta}{2} \psi(n)-\sum\limits_{k=n+1}^\infty \psi(k)\Big),
\end{equation}
where $\delta$ satisfy the condition \eqref{10'}, and therefore, by virtue of the belonging $f_{n,p,\bar{\beta}}(\delta,\cdot)\in C_{\bar\beta,1}^\psi,$ the following inequality is true 
\begin{equation}\label{14'}
E_n(C_{\bar\beta,1}^\psi)_C\ge\frac{1}{\pi}\Big(\frac2\delta\sin\frac{\delta}{2} \psi(n)-\sum\limits_{k=n+1}^\infty \psi(k)\Big).
\end{equation}
Taking the limit as  $\delta\rightarrow0$ in the inequality \eqref{14'}, we obtain  \eqref{3'}.
\end{proof}

The main statement of the work is the following theorem.

\begin{theorem}\label{2t}
	For arbitrary $\{\beta_k\}_{k=1}^\infty,$ $\beta_k\in\mathbb R,$ $1 < p \le \infty,$   $n\in\mathbb{N}$ and $\psi(k)\ge0$, which satisfy the condition \eqref{2}, the following inequalities hold
\begin{multline}\label{12}
	\frac{\|\cos t\|_{p'}}{\pi}\Big(\psi(n)-\sum\limits_{k=n+1}^\infty \psi(k)\Big)	\le\\ \le E_{n}(C_{\bar\beta,p}^\psi)_{C}\le {\cal E}_{n}(C_{\bar\beta,p}^\psi)_{C}\le\\
	\le\frac{\|\cos t\|_{p'}}{\pi}\Big(\psi(n)+\sum\limits_{k=n+1}^\infty \psi(k)\Big),
\end{multline}
	where $\displaystyle\frac1p+\frac1{p'}=1.$\\
	If  $\psi(k)$ satisfies the condition \eqref{15_4'}, then the following asymptotic equalities hold
	\begin{equation}\label{14_14}
	\left.\begin{array}{c}
	{\cal E}_{n}(C_{\bar\beta,p}^\psi)_{C}\\ E_{n}(C_{\bar\beta,p}^\psi)_{C}
	\end{array}\right\}=
	\frac{\|\cos t\|_{p'}}{\pi}\psi(n)+{\cal O}(1)\sum\limits_{k=n+1}^\infty \psi(k),
	\end{equation}
	where ${\cal O}(1)$ are the quantities uniformly bounded in all parameters.
\end{theorem}

Note that the asymptotic equality \eqref{14_14} for the quantities ${\cal E}_{n}(C_{\bar\beta,p}^\psi)_{C}$ is established in paper \cite{Serdyuk_2005_8}.

\begin{proof} Taking into account Theorems 1 and 2, it suffices to verify the validity of the last inequality in $(\ref{12})$.

By virtue of H$\rm\ddot{o}$lder's inequality for arbitrary $f\in C_{\bar{\beta},p}^\psi, 1\le p\le\infty,$
$$
\|f-S_{n-1}(f)\|_C\le \frac{1}{\pi}\|\varphi\|_p\left\|\sum\limits_{k=n}^\infty \psi(k)\cos\left(nt-\frac{\beta_k\pi}2\right)\right\|_{p'}\le
$$
$$
\le \frac{1}{\pi}\sum\limits_{k=n}^\infty \psi(k)\left\|\cos\left(nt-\frac{\beta_k\pi}2\right)\right\|_{p'}=
\frac{\|\cos t\|_{p'}}{\pi}\sum\limits_{k=n}^\infty \psi(k).
$$ 
Thus,
\begin{equation}
{\cal E}_{n}(C_{\bar\beta,p}^\psi)_{C}\le\frac{\|\cos t\|_{p'}}{\pi}\sum\limits_{k=n}^\infty \psi(k).
\label{16}
\end{equation}
Inequalities \eqref{12} follow from  \eqref{3}, \eqref{3'} and \eqref{16}.

To verify the validity of the asymptotic equalities \eqref{14_14} just go to the limit as $n\rightarrow\infty$ in (\ref{12}) and take into account the condition \eqref{15_4'}. 
\end{proof} 

Note that the condition \eqref{15_4'} is satisfied if  $\psi(k)$  satisfies the condition $D_0:$
\begin{equation}\label{17}
\lim_{k\rightarrow\infty}\frac{\psi(k+1)}{\psi(k)знам}=0.
\end{equation}

We give the corollaries of Theorem 3 in some important special cases.


\begin{theorem}\label{3t}
	Let $1 \le p \le \infty,$  $r>1$, $n\in\mathbb{N}$  and $\{\beta_k\}_{k=1}^\infty$ be an arbitrary sequence of 	real numbers. Then for $r\ge n+1$ such that
	\begin{equation}\label{18}
	\left(1+\frac1n\right)^{-r}<\left(2+\frac1n\right)^{-1}
	\end{equation}
	the following inequalities hold
	\begin{multline}\label{19}
	\frac{\|\cos t\|_{p'}}{\pi} n^{-r}\left(1-\frac{2+\frac1n}{\left(1+\frac1n\right)^r}\right)	\le\\ \le E_{n}(W_{\bar\beta,p}^r)_{C}\le {\cal E}_{n}(W_{\bar\beta,p}^r)_{C}\le\\ \le
	\frac{\|\cos t\|_{p'}}{\pi} n^{-r}\left(1+\frac{2+\frac1n}{\left(1+\frac1n\right)^r}\right),
	\end{multline}
	where $\displaystyle\frac1p+\frac1{p'}=1.$
\end{theorem}

\begin{proof} Put  $\psi(k)=k^{-r}, r>1,$ and make sure that the condition \eqref{2}  is folowed from the inequality \eqref{18}.

Since for arbitrary   $n\in\mathbb{N}$ and $r\ge n+1$
\begin{equation}\label{20}
\sum_{k=n+1}^\infty\frac1{k^r}<\frac1{(n+1)^r}+\int\limits_{n+1}^\infty\frac{dt}{t^r}=\frac1{(n+1)^r}+\frac1{(r-1)(n+1)^{r-1}}=
$$
$$
=\frac1{(n+1)^r}\frac{r+n}{r-1знам}\le\frac1{n^r}\frac1{(1+\frac1n)^r}\frac{2r-1}{r-1знам}
\le\frac1{n^r}\frac1{(1+\frac1n)^r}\left(2+\frac1n\right),
\end{equation}
then  under the condition \eqref{18} we obtain 
$$
\sum_{k=n+1}^\infty\frac1{k^r}<\frac1{n^r}.
$$
Then, applying Theorem 3 for $\psi(k)=k^{-r}$, from (\ref{12}) and \eqref{20} we obtain \eqref{19}. \end{proof}

Note that if the condition
\begin{equation}\label{21}
\lim\limits_{n\rightarrow\infty}\frac rn=\infty,
\end{equation}
is satisfied, then the condition \eqref{18} is also satisfied for sufficiently large $n$, because
$$
\left(1+\frac1n\right)^{-r}=\left(\left(1+\frac1n\right)^{n+1}\right)^{-\frac{r}{n+1}}\le e^{-\frac r{n+1}}\rightarrow0,\quad n\rightarrow\infty.
$$
And therefore
\begin{equation*}
\frac{2+\frac1n}{\left(1+\frac1n\right)^r}\rightarrow0 \quad \mbox{при}\quad  n\rightarrow\infty.
\end{equation*}
Taking the limit as  $n\rightarrow\infty$ in the inequalities \eqref{19}, we obtain asymptotic equalities for the quantities $E_{n}(W_{\bar\beta,p}^r)_{C}$ and ${\cal E}_{n}(W_{\bar\beta,p}^r)_{C}$.

\begin{theorem}\label{4t}
	Let $1 \le p \le \infty,$  $r>1$, $n\in\mathbb{N}$  and $\{\beta_k\}_{k=1}^\infty$ be an arbitrary sequence of 	real numbers. If the condition \eqref{21} is satisfied, then the following asymptotic equalities hold
	\begin{equation}\label{14_4t}
	\left.\begin{array}{c}
	{\cal E}_{n}(W_{\bar\beta,p}^r)_{C}\\ E_{n}(W_{\bar\beta,p}^r)_{C}
	\end{array}\right\}=\frac1{n^r}
	\left(\frac{\|\cos t\|_{p'}}{\pi}+{\cal O}(1)\left(1+\frac1n\right)^{-r}\right),
	\end{equation}
	where 	$\displaystyle\frac1p+\frac1{p'}=1$ and ${\cal O}(1)$ are the quantities uniformly bounded in all parameters.
\end{theorem}

The asymptotic equality \eqref{14_4t} for the quantities ${\cal E}_{n}(C_{\bar\beta,p}^\psi)_{C}$ is established in  \cite{Serdyuk_Sokolenko2019}.

As it showed in \cite[P. 163-164]{Stepanets1987}
\begin{equation}\label{24}
\sum_{k=n+1}^{\infty}e^{-\alpha k^r}<e^{-\alpha n^r}\left(1+\frac1{\alpha r n^{r-1}}\right)e^{-\alpha r n^{r-1}}, \quad r>1, \alpha>0, n\in\mathbb{N}.
\end{equation}

So, if the condition
\begin{equation}\label{25_25}
\left(1+\frac1{\alpha r n^{r-1}}\right)e^{-\alpha r n^{r-1}}<1, \quad r>1, \alpha>0, 
\end{equation}
is satysfied, then
$$
\sum_{k=n+1}^{\infty}e^{-\alpha k^r}<e^{-\alpha n^r},
$$
and therefore from Theorem 3 for the classes $C_{\bar\beta,p}^{\alpha,r}$ we obtain the following statement.

\begin{theorem}\label{5t}
	Let $1 \le p \le \infty,$ $r>1,$ $\alpha>0,$ $n\in\mathbb{N}$ and $\{\beta_k\}_{k=1}^\infty$   be an arbitrary sequence of real numbers. If the inequality \eqref{25_25} is satisfied, then the following relations hold
	\begin{equation*}\label{26}
	\frac{\|\cos t\|_{p'}}{\pi} e^{-\alpha n^r}\left(1-\left(1+\frac1{\alpha r n^{r-1}}\right)e^{-\alpha r n^{r-1}}\right)	\le E_{n}(C_{\bar\beta,p}^{\alpha,r})_{C}\le
	$$
	$$
	\le {\cal E}_{n}(C_{\bar\beta,p}^{\alpha,r})_{C}\le
	\frac{\|\cos t\|_{p'}}{\pi} e^{-\alpha n^r}\left(1+\left(1+\frac1{\alpha r n^{r-1}}\right)e^{-\alpha r n^{r-1}}\right),
	\end{equation*}
	where $\displaystyle\frac1p+\frac1{p'}=1.$\\
\noindent	As $n\rightarrow\infty$ the following asymptotic equalities hold	
	\begin{equation}\label{14_5t}
\!\!\!\!	\left.\begin{array}{c}
	{\cal E}_{n}(C_{\bar\beta,p}^{\alpha,r})_{C}\\
	E_{n}(C_{\bar\beta,p}^{\alpha,r})_{C}
	\end{array}\right\}{=}
	e^{{-}\alpha n^r}\left(\frac{\|\cos t\|_{p'}}{\pi}{+}{\cal O}(1)\left(1{+}\left(1{+}\frac1{\alpha r n^{r{-}1}}\right)e^{{-}\alpha r n^{r-1}}\right)\right),
	\end{equation}
	where ${\cal O}(1)$ are the quantities uniformly bounded in all parameters.
\end{theorem}

Note that for $p=\infty$ the asymptotic equality \eqref{14_5t} for the quantities ${\cal E}_{n}(C_{\bar\beta,p}^{\alpha,r})_{C}$ was obtained by Stepanets \cite{Stepanets1987} and for $1\le p<\infty$  in  \cite{Serdyuk_2005_8}.

\bigskip

CONTACT INFORMATION

\medskip
A.S.~Serdyuk\\01024, Ukraine, Kiev-4, 3, Tereschenkivska st.\\
serdyuk@imath.kiev.ua

\medskip
I.V.~Sokolenko\\01024, Ukraine, Kiev-4, 3, Tereschenkivska st.\\
sokol@imath.kiev.ua
\end{document}